\documentclass[a4paper,12pt]{article}

\newenvironment{proof}{\noindent {\bf Proof:}}{\hfill $\Box$}

\newtheorem{theorem}{Theorem}
\newtheorem{lemma}{Lemma}

\newtheorem{corollary}{Corollary}

\newtheorem{assumption}{Assumption}

\usepackage{booktabs}

\usepackage{amsmath}
\usepackage{amssymb}
\usepackage{amsfonts}
\usepackage{graphicx}

\usepackage{color}

\DeclareMathOperator*{\argmin}{arg\,min}

\title{\bf Convergence rates of moment-sum-of-squares hierarchies for optimal control problems}

\begin{document}

\author{Milan Korda$^1$, Didier Henrion$^{2,3,4}$, Colin N. Jones$^1$}

\footnotetext[1]{Laboratoire d'Automatique, \'Ecole Polytechnique F\'ed\'erale de Lausanne, Station 9, CH-1015, Lausanne, Switzerland. {\tt \{milan.korda,colin.jones\}@epfl.ch}}
\footnotetext[2]{CNRS; LAAS; 7 avenue du colonel Roche, F-31400 Toulouse; France. {\tt henrion@laas.fr}}
\footnotetext[3]{Universit\'e de Toulouse; LAAS; F-31400 Toulouse; France.}
\footnotetext[4]{Faculty of Electrical Engineering, Czech Technical University in Prague,
Technick\'a 2, CZ-16626 Prague, Czech Republic.}

\date{Draft of \today}

\maketitle

\begin{abstract}
We study the convergence rate of moment-sum-of-squares hierarchies of semidefinite programs for optimal control problems with polynomial data. It is known that these hierarchies generate polynomial under-approximations to the value function of the optimal control problem and that these under-approximations converge in the $L^1$ norm to the value function as their degree $d$ tends to infinity. We show that the rate of this convergence is $O(1/\log\log\,d)$. We treat in detail the continuous-time infinite-horizon discounted problem and describe in brief how the same rate can be obtained for the finite-horizon continuous-time problem and for the discrete-time counterparts of both problems.
\end{abstract}

\begin{center}\small
\end{center}

\section{Introduction}

The moment-sum-of-squares hierarchy (also know as Lasserre hierarchy) of semidefinite programs was originally introduced in \cite{l00} in
the context of polynomial optimization. It allows one to solve globally non-convex optimization problems at the price
of solving a sequence, or hierarchy, of convex semidefinite programming problems, with convergence guarantees; see e.g. \cite{l09} for an introductory survey, \cite{l10} for a comprehensive overview and \cite{c10} for control applications.

This hierarchy was extended in \cite{lhpt08} to polynomial optimal control, and later on in \cite{hls09} to global
approximations of semi-algebraic sets, originally motivated by volume and integral estimation problems. The approximation hierarchy
for semi-algebraic sets derived in \cite{hls09} was then transposed and adapted to an approximation hierarchy
for transcendental sets relevant for systems control \cite{c11}, such as regions of attraction \cite{hk14} and maximal invariant sets
for controlled polynomial differential and difference equations \cite{khj14}, still with rigourous analytic convergence guarantees.

Central to the moment-sum-of-squares hierarchies of \cite{lhpt08,hk14,khj14} are polynomial subsolutions of the Hamilton-Jacobi-Bellman equation,
providing certified lower bounds, or under-approximations, of the value function of the optimal control problem.
It was first shown in \cite{lhpt08} that the hierarchy of polynomial subsolutions of increasing degree converges locally
(i.e. pointwise) to the value function on its domain. Later on, as an outcome of the results of \cite{hk14}, global convergence
(i.e. in $L^1$ norm on compact domains, or equivalently, almost uniformly) was established in \cite{hp14}.

The current paper is motivated by the analysis of the rate of convergence of the moment-sum-of-squares hierarchy for static
polynomial optimization achieved in \cite{ns07}. We show that a similar analysis can be carried out in the dynamic
case, i.e. for assessing the rate of convergence of the moment-sum-of-squares hierarchy for polynomial optimal control.
For ease of exposition, we focus on the discounted infinite-horizon continuous-time optimal control problem and briefly describe (in Section~\ref{sec:extension}) how the same convergence rate can be obtained for the finite-time continuous version of the problem and for the discrete counterparts of both problems.

Our main Theorem~\ref{mainresult} gives estimates on the rate of convergence of the polynomial under-approximations
to the value function in the $L^1$ norm. As a direct outcome of this result, we derive in Corollary~\ref{rate} that the
rate of convergence is in $O(1/\log\log\,d)$, where $d$ is the degree of the polynomial approximation. As far as we know, this is the first estimate of this kind in the context of moment-sum-of-squares hierarchies for
polynomial optimal control.


\subsection{Notation}
The set of all continuous functions on a set $X\subset \mathbb{R}^n$ is denoted by $C(X)$;  the set of all $k$-times continuously differentiable functions is denoted by $C^k(X)$. For $h \in C(X)$, we denote $\| h \|_{C^0(X)} := \max_{x \in X} | h(x) |$ and for $h \in C^1(X)$ we denote $\| h\|_{C^1(X)} := \max_{x \in X} | h(x) | + \max_{x \in X} \|  \nabla h(x) \|_2$ where $\nabla h$ is the gradient of $h$. The $L^1$ norm with respect to a measure $\mu_0$ of a measurable function $h:\mathbb{R}^n\to \mathbb{R}$ is denoted by $\| h\|_{L_1(\mu_0)}:=\int_{\mathbb{R}^n} h(x)\mu_0(dx)$. The set of all multivariate polynomials in a variable $x$ of total degree no more than $d$ is denoted by~$\mathbb{R}[x]_d$. The symbol $\mathbb{R}[x]_d^n$ denotes the $n$-fold cartesian product of this set, i.e., the set of all vectors with $n$ entries, where each entry is a polynomial from $\mathbb{R}[x]_d$. The interior of a set~$X \subset \mathbb{R}^n$ is denoted by $\mathrm{Int}(X)$.

\section{Problem setup}

Consider the discounted infinite-horizon optimal control problem
\begin{equation}\label{ocp_basic}
\begin{array}{rclll}
V^\star(x_0) & := & \inf\limits_{u(\cdot),\,x(\cdot)} &  \int_0^\infty e^{-\beta t} l(x(t),u(t))\,dt   \\
&& \hspace{0.6cm} \mathrm{s.t.}  & x(t) = x_0 + \int_0^t f(x(s),u(s))\,ds\,  &\forall t\in [0,\infty)\\
&&& x(t) \in X, \; u(t) \in U  &\forall t\in [0,\infty)
\end{array}
\end{equation}
where $\beta > 0$ is a given discount factor, $f \in \mathbb{R}[x,u]^n_{d_f}$ and $l \in \mathbb{R}[x,u]_{d_l}$
are given multivariate polynomials and the state and input constraint sets $X$ and $U$ are of the form
\[
X = \{ x \in \mathbb{R}^n \mid g_i^X(x) \ge 0, i = 1,\ldots,n_{X} \},
\]
\[
U = \{ u \in \mathbb{R}^m \mid g_i^U(u) \ge 0, i = 1,\ldots,n_{U} \},
\]
where $g_i^X \in \mathbb{R}[x]_{d_i^X}$ and $g_i^U \in \mathbb{R}[u]_{d_i^U}$ are multivariate polynomials. The function $V^*$ in (\ref{ocp_basic}) is called the \emph{value function} of the optimal control problem~(\ref{ocp_basic}).

Let us recall the Hamilton-Jacobi-Bellman inequality
\begin{equation}\label{eq:bellman}
l(x,u) - \beta V(x,u) + \nabla V(x,u) \cdot f(x,u) \ge 0 \quad \forall\,(x,u) \in X\times U
\end{equation}
which plays a crucial role in the derivation of the convergence rates. In particular, for any function $V \in C^1(X)$ that satisfies~(\ref{eq:bellman}) it holds
\begin{equation}\label{eq:lowBoundBasic}
V(x) \le V^\star (x) \quad \forall\, x \in X.
\end{equation}
The following polynomial sum-of-squares optimization problem provides a  sequence of lower bounds to the value function indexed by the degree $d$:
\begin{equation}\label{opt:lb_sos}
\begin{array}{rclll}
 & \max\limits_{V \in \mathbb{R}[x]_d} &  \int_{X} V (x) \, d\mu_0(x)  \\
& \hspace{0.0cm} \mathrm{s.t.}  & l - \beta V + \nabla V \cdot f  \in Q_{d+d_f}(X\times U), \vspace{1mm}
\end{array}
\end{equation}
where $\mu_0$ is a given probability measure supported on $X$ (e.g., the uniform distribution), and
\begin{align*}
Q_{d+d_f}(X\times U)  := \Big\{  & s_0 + \sum_{i=1}^{n_{X}} g_i^X s_X^i + \sum_{i=1}^{n_{U}} g_i^U s_U^i   \: :\: \\
& s_0\in \Sigma_{\lfloor(d+d_f)/2\rfloor  } , s_X^i \in \Sigma_{\lfloor(d + d_f - d_X^i)/2\rfloor  },  
s_U^i \in \Sigma_{\lfloor(d + d_f - d_U^i)/2\rfloor  }
\Big\},
\end{align*}
is the truncated quadratic module associated with the sets $X$ and $U$ (see \cite{l09} or \cite{l10}), where $\Sigma_d$ is the cone of sums of squares of polynomials of degree up to $d$. Note that whenever $V$ is feasible in~(\ref{opt:lb_sos}), then $V$ satisfies Bellman's inequality (\ref{eq:bellman}), because polynomials in $Q_{d+d_f}(X\times U)$ are non-negative on $X\times U$ by construction. Therefore any polynomial $V$ feasible in~(\ref{opt:lb_sos}) satisfies also~(\ref{eq:lowBoundBasic}) and hence is an under-approximation of $V^\star$ on $X$. 

The truncated quadratic module is essential to the proof of convergence of the moment-sum-of-squares hieararchy in the static polynomial optimization case~\cite{l00} which is based on Putinar's Positivstellensatz~\cite{putinar93}. We recall that some polynomials of degree $d+d_f$ non-negative
on $X\times U$ may not belong to $Q_{d+d_f}(X\times U)$~\cite{l10}. On the other hand, optimizing over the polynomials belonging to $Q_{d+d_f}(X\times U)$ is ``simple'' (it translates to semidefinite programming) while optimizing over the cone of non-negative polynomials is very difficult in general. In particular, the optimization problem~(\ref{opt:lb_sos}) translates to a finite-dimensional semidefinite programming problem (SDP). The fact that the truncated quadratic module has an explicit SDP representation and hence can be tractably optimized over is one of the main reasons for the popularity of the moment-sum-of-squares hierarchies across many fields of science.

Throughout the paper we impose the following standing assumptions.
\begin{assumption}\label{as:main} The following conditions hold:
\begin{enumerate}
  \item[(a)] $X \subset [-1,1]^{n}$ and $U \subset [-1,1]^{m}$.
   \item[(b)] The sets of polynomials $(g_i^X)_{i=1}^{n_X}$ and $(g_i^U)_{i=1}^{n_U}$ both satisfy the Archimedian condition\footnote{A sufficient condition for a set of polynomials $(g_i)_{i=1}^n$ to satisfy the Archimedian condition is $g_i = N - \| x\|_2^2$ for some $i$ and some $N \ge 0$, which is a non-restrictive condition provovided that the set defined by $g_i$'s is compact and an estimate of its dimeter is known. For a precise definition of this condition see Section 3.6.2 of~\cite{l09}.}.
   \item[(c)] $0 \in \mathrm{Int}(X)$ and $0 \in \mathrm{Int}(U)$.
   \item[(d)]  The function $\nabla V^\star$ is Lipschitz continuous on $X$.
   \item [(e)] The set $f(x,U)$ is convex for all $x \in X$ and the function $v\mapsto \inf\limits_{u\in U}\{ l(x,u) \mid v = f(x,u)  \}$ is convex for all $x \in X$.
 \end{enumerate}
\end{assumption}

The Assumption $(a)$ and $(b)$ are made without loss of generality since the sets $X$ and $U$ are assumed to be compact and hence can be scaled such that they are included in the unit ball; adding redundant ball constraints $1-\|x\|^2$ and $1-\|u\|^2$ in the description of $X$ and $U$ then implies the Archimedianity condition. Assumption $(c)$ essentially requires that the sets $X$ and $U$ have nonempty interiors (a mild assumption) since then a change of coordinates can always be carried out such that the origin is in the interior of these sets. Assumption $(d)$ is an important regularity assumption necessary for the subsequent developments. Assumption $(e)$ is a standard assumption ensuring that the value function of the so-called relaxed formulation of the problem~(\ref{opt:lb_sos}) coincides with $V^\star$ (see, e.g., \cite{vinter1993}) and is satisfied, e.g., for input-affine\footnote{A system is input-affine if $f(x,u) = f_x(x) + f_u(x) u$ for some functions $f_x$ and $f_u$.} systems with input-affine cost function provided that $U$ is convex. This class of problems is by far the largest and practically most relevant for which this assumption holds although other problems exist that satisfy this assumption as well\footnote{For example, consider $l(x,u) = x^2$, $f(x,u) = x + u^2$, $U = [-1,1]$.}.

Under Assumption~\ref{as:main}, the hierarchy of lower bounds generated by problem~(\ref{opt:lb_sos}) converges from below in the $L^1$ norm to the value function~$V^\star$; see e.g. \cite{hp14}:

\begin{theorem}
There exists a $d_0 \ge 0$ such that the problem~(\ref{opt:lb_sos}) is feasible for all $d\ge d_0$. In addition $V \le V^\star$ for any $V$ feasible in~(\ref{opt:lb_sos}) and $\lim_{d\to \infty}   \|V^\star - V_d^\star  \|_{L_1(\mu_0)}  = 0$, where $V_d^\star$ is an optimal solution to~(\ref{opt:lb_sos}).
\end{theorem}

The goal of this paper is to derive bounds on the convergence rate of $V_d^\star$ to $V^\star$.

\section{Convergence rate}

The convergence rate is a consequence of the following fundamental results from approximation theory and polynomial optimization.
\begin{theorem}[Bagby et al.~\cite{bagby}]\label{thm:bagby}
If $h :  X \to \mathbb{R}$ is a function such that $\nabla h \in C^1(X)$, then there exists a sequence of polynomials $(p_d)_{d = 1}^\infty$ satisfying $\mathrm{deg}(p_d) \le d$ such that $\| h - p_d \|_{C^1(X)} \le c_1 / d$ for some constant $c_1 \ge 0$ depending on $h$ and $X$ only.
\end{theorem}

Now we turn to the second fundamental result. Given a polynomial $p \in  \mathbb{R}[x]_d$ expressed in a multivariate monomial basis as 
\[
 p(x) = \sum_{\substack{\alpha \in \mathbb{N}^n\\
                  | \alpha | \le d}} \beta_\alpha x^\alpha
\]
with $|\alpha| = \sum_{i=1}^n \alpha_i$ and $x^\alpha = \prod_{i=1}^n x^{\alpha_i} $, we define
\begin{equation}\label{eq:Rnorm}
\| p\|_{\mathbb{R}[x]} = \max_{\alpha} \frac{|\beta_\alpha|}{\binom{|\alpha|}{\alpha}},
\end{equation}
where the multinomial coefficient $\binom{|\alpha|}{\alpha}$ is defined by
\[
\binom{|\alpha|}{\alpha} := \frac{| \alpha|!}{\alpha_1 ! \cdot \ldots \cdot \alpha_n!}.
\]

\begin{theorem}[Nie \& Schweighofer ~\cite{ns07}]\label{thm:nie}
Let $p \in \mathbb{R}[x,u]_{d_p}$ and let
\[
p_{\mathrm{min}} :=  \min_{(x,u)\in X\times U}p(x,u)  \quad \text{with}\quad  p_{\mathrm{min}} > 0.
\]
Then $p \in Q_d(X\times U)$ provided that
\begin{equation}\label{eq:schweigBound}
d \ge c_2 \, \mathrm{exp} \Big ( d_p^2 (n+m)^{d_p}\frac{\| p\|_{\mathbb{R}[x,u]}}{p_\mathrm{min}}  \Big)^{c_2}, 
\end{equation}
where the constant $c_2$ depends only on the sets $X$ and $U$.
\end{theorem}

In the following developments it will be crucial to bound the norm $\| \cdot\|_{\mathbb{R}[x]}$ of a polynomial by its supremum norm $\| \cdot\|_{C(X)}$. We remark that such a bound is possible only for a ``generic'' set $X$ such that any polynomial vanishing on $X$ necessarily vanishes everywhere. A sufficient condition for this is $\mathrm{Int}(X) \neq \emptyset$. This is the reason for Assumption~\ref{as:main} $(c)$.
%

\begin{lemma}\label{lem:normBound1}
If $p \in \mathbb{R}[x]_d$, $x\in \mathbb{R}^n$, then 
\begin{equation}\label{eq:norm_eq}
\|p \|_{\mathbb{R}[x]} \le 3^{d+1} \| p \|_{C([-1,1]^n)}
\end{equation}
for all $d \ge 0$.
\end{lemma}
\begin{proof}
The idea is to use a multivariate Markov inequality to bound the derivatives of the polynomial at zero (and hence its coefficients) in terms of its supremum norm on $[-1,1]^n$.

Let $p = \sum_{\alpha} \beta_\alpha x^\alpha \in \mathbb{R}[x]_d$. From \cite[Theorem~6]{skalyga}, we have
\[
 \Big | \frac{\partial^{|\alpha|} p}{\partial x^\alpha}(0) \Big|  \le  |  T_d^{(|\alpha|)}(0) + i S_d^{(|\alpha|)}(0) | \cdot \| p\|_{C([-1,1]^n)}
\]
for all multiindices $\alpha$ satisfying $|\alpha| \le d$, where $i = \sqrt{-1}$, $T_d(y) = \cos(d \arccos(y))$, $y\in[-1,1]$, denotes the $d$-th univariate Chebyshev polynomial of the first kind, $S_d(y) = \sin(d \arccos(y)) = d^{-1}\sqrt{1-x^2}T_d'(y)$, $y \in [-1,1]$, and $h^{(k)}$ signifies the $k$-th derivative of a function $h : \mathbb{R}\to \mathbb{R}$. It is easy to see that $S_d^{(k)}(0) = d^{-1}T_d^{(k+1)}(0)$ and hence
\begin{align*}
|  T_d^{(|\alpha|)}(0) + i S_d^{(|\alpha|)}(0) |  &\le  |  T_d^{(|\alpha|)}(0) | +  \frac{1}{d}|  T_d^{(|\alpha|+1)}(0) | = |\alpha|! \cdot |t_{d,|\alpha|}| + \frac{(|\alpha|+1)!}{d} \cdot |t_{d,|\alpha|+1}| \\  &\le \Big[|\alpha|! + \frac{(|\alpha|+1)!}{d}\Big]\bar{t}_{d},
\end{align*}
where $t_{d,k}$ denotes the $k$-th coefficient of $T_d$ when expressed in the monomial basis (i.e., $T_d(y) = \sum_{k=0}^d t_{d,k} y^k$) and $\bar{t}_d = \max_{k\in\{0,\ldots,d\}} |t_{d,k}|$.

Since $\beta_\alpha = (\alpha_1 ! \cdot \ldots \cdot \alpha_n!)^{-1} \frac{\partial^{|\alpha|} p}{\partial x^\alpha}(0)$, we get
\begin{align*}
\frac{|\beta_\alpha|}{\binom{|\alpha|}{\alpha}} &= \frac{(\alpha_1 ! \cdot \ldots \cdot \alpha_n!)|\beta_\alpha|}{|\alpha|!} = \frac{1}{|\alpha|!}\Big | \frac{\partial^{|\alpha|} p}{\partial x^\alpha}(0) \Big| \le \Big[1 +\frac{|\alpha|+1}{d} \Big] \bar{t}_d \| p\|_{C([-1,1]^n)}.
\end{align*}
In view of~(\ref{eq:Rnorm}) and since $|\alpha| \le d$ we get
\[
\| p \|_{\mathbb{R}[x]} \le \Big[2 + \frac{1}{d}\Big] \bar{t}_d \| p\|_{C([-1,1]^n)}.
\]
It remains to bound $\bar{t}_d$. From the generating recurrence of $T_{d+1}(y) = 2y T_d(y) - T_{d-1}(y) $ starting from $T_0 = 1$ and $T_1 = y$, it follows that $\bar{t}_d \le \tilde{t}_d$, where $\tilde{t}_d$ solves the linear difference equation
\[
\tilde{t}_{d+1} = 2\tilde{t}_d + \tilde{t}_d
\]
 with the initial condition $\tilde{t}_0 = 1$ and $\tilde{t}_1 = 1$. The solution to this equation is
 \[
 \tilde{t}_d = (1+\sqrt{2})^d\Big(\frac{\sqrt{2}}{2} + \frac{1}{2}\Big) + (1-\sqrt{2})^d\Big(\frac{1}{2} - \frac{\sqrt{2}}{2}\Big) \le 3^d,\quad d\ge 1.
 \]
Therefore $\bar{t}_d \le 3^d$ for $d \ge 1$ and hence
\[
\| p \|_{\mathbb{R}[x]} \le \Big[2 + \frac{1}{d}\Big] 3^d \| p\|_{C([-1,1]^n)} \le 3^{d+1} \| p\|_{C([-1,1]^n)},\quad d\ge 1.
\]
Since $\| p \|_{\mathbb{R}[x]} =  \| p\|_{C([-1,1]^n)}$ for $d = 0$, the result follows.
\end{proof}

In order to state an immediate corollary of this result, crucial for subsequent developments, we define
\begin{equation}\label{eq:rdef}
 r := \frac{1}{\sup\{s > 0 \mid [-s,s]^{n+m} \subset X\times U \}},
 \end{equation}
which is the reciprocal value of the length of the side of the largest box centered around the origin included in $X\times U$. By Assumption~\ref{as:main} $(a)$ and $(c)$, we have $r \in [1,\infty)$.

\begin{corollary}\label{cor:norms}
If $p \in \mathbb{R}[x,u]_d$, then
\begin{equation}\label{eq:norm_eq}
\|p \|_{\mathbb{R}[x,u]} \le k(d) \| p \|_{C(X\times U)},
\end{equation}
where $k(d) = 3^{d+1} r^d $ with $r$ defined in~(\ref{eq:rdef}).

\end{corollary}
\begin{proof}
Set $\tilde{p}((x,u)) := p(r^{-1}(x,u))$. Then we have
\begin{equation}\label{eq:aux1}
\| \tilde{p}\|_{C([-1,1]^{n+m})} = \| p \|_{C([-1/r,1/r]^{n+m})} \le \| p\|_{C(X\times U)}
\end{equation}
since $[-1/r,1/r]^{n+m} \subset X\times U$ by definition of $r$~(\ref{eq:rdef}). In addition
\begin{equation}\label{eq:aux2}
\| \tilde{p}\|_{\mathbb{R}[x,u]} = \max_{\alpha} r^{-|\alpha|}\frac{|\beta_\alpha|}{\binom{|\alpha|}{\alpha}} = r^{-d}\max_{\alpha} r^{d-|\alpha|}\frac{|\beta_\alpha|}{\binom{|\alpha|}{\alpha}} \ge r^{-d}\| p \|_{\mathbb{R}[x,u]}.
\end{equation}
Combining (\ref{eq:aux1}), (\ref{eq:aux2}) and Lemma~\ref{lem:normBound1} we get
\[
\| p\|_{\mathbb{R}[x,u]} \le r^d \| \tilde p\|_{\mathbb{R}[x,u]} \le 3^{d+1} r^d\| \tilde p\|_{C([-1,1]^{n+m})} \le 3^{d+1} r^d  \|p\|_{C(X\times U)} = k(d) \|p\|_{C(X\times U)}.
\]
as desired.
\end{proof}

Now we turn to analyzing the Bellman inequality~(\ref{eq:bellman}). The following immediate property of this inequality will be of importance:
\begin{lemma}\label{lem:shift}
Let $V$ satisfy (\ref{eq:bellman}) and let $a \in \mathbb R$. Then $\tilde{V} := V - a$ satisfies
\[
l - \beta \tilde{V} + \nabla \tilde{V} \cdot f \ge \beta a \quad \forall\,(x,u) \in X\times U.
\]
\end{lemma}
\begin{proof}
We have $l - \beta \tilde{V} + \nabla \tilde{V} \cdot f = l - \beta V + \nabla V \cdot f + \beta a \ge \beta a$, 
since $V$ satisfies~(\ref{eq:bellman}).
\end{proof}

We will also need a result which estimates the distance between the best polynomial approximation of a given degree to the value function and polynomials of the same degree satisfying Bellman's inequality. A similar result in discrete time and with discrete state and control spaces can be found in~\cite{vanRoy}.
\begin{lemma}\label{lem:dp_bound}
Let
\[
\hat{V}_d \in \argmin_{V \in \mathbb{R}[x]_d} \| V - V^\star \|_{C^1(X)}.
\]
Then there exists a polynomial $\tilde{V}_d \in \mathbb{R}[x]_d$ satisfying (\ref{eq:bellman}) and such that
\begin{equation}\label{eq:Vtilde_bound}
\| \tilde{V}_d - V^\star   \|_{C^1(X)} \le \| \hat{V}_d - V^\star   \|_{C^1(X)} \Big(2 + \frac{\| f\|_{C^0(X)}}{\beta}\Big).
\end{equation}
\end{lemma}
\begin{proof}
Let $\tilde{V}_d := \hat{V}_d - a$. We will find an $a \ge 0$ such that $\tilde{V}_d$ satisfies the Bellman inequality. We have
\begin{align*}
l - \beta \tilde{V}_d + \nabla \tilde{V}_d f &= l - \beta \hat{V}_d + \nabla \hat{V}_d f + \beta a \\ &= 
l - \beta V^\star + \nabla V^\star f + \beta(V^\star - \hat{V}_d) + (\nabla\hat{V}_d -\nabla V^\star)f   + \beta a \\
& \ge \beta(V^\star-\hat{V}_d)+ (\nabla\hat{V}_d -\nabla V^\star)f   + \beta a \\
& \ge -\beta\| \hat{V}_d - V^\star\|_{C^1(X)} - \| \hat{V}_d - V^\star\|_{C^1(X)}\|f\|_{C^0(X)} + \beta a,
\end{align*}
and hence if
\[
a := \Big(1 + \frac{\|f \|_{C^0(X)}}{\beta}\Big)\| \hat{V}_d - V^\star\|_{C^1(X)},
\]
then $\tilde{V}_d$ satisfies Bellman's inequality  (\ref{eq:bellman}) and estimate~(\ref{eq:Vtilde_bound}) holds.
\end{proof}

Now we are in position to prove our main result which bounds the gap, in $L^1$ norm, between the
value function $V^\star$ of the optimal control problem (\ref{ocp_basic}) and any optimal solution $V^{\star}_d$ of the sum-of-squares program (\ref{opt:lb_sos}):
\begin{theorem}\label{mainresult}
It holds that $ \|V^\star - V_d^\star \|_{L_1(\mu_0)} < \epsilon $ for all
integer \begin{align}
d  &\ge   c_2 \cdot \mathrm{exp}\Big[     \Big ( 6d_p(\epsilon)^2 (3r(n+m))^{d_{p}(\epsilon)}\frac{M + \beta\epsilon + \delta_1 \|f \|_{C^0}}{\beta\epsilon}  \Big)^{c_2}     \Big] \label{eq:nonasympt} \\ & =
O\Big(  \mathrm{exp}\big[\frac{1}{\epsilon^{3c_2}}(3(n+m)r)^{\frac{c_3}{\epsilon}}  \big]\Big), \label{eq:asympt}
\end{align}
where $d_p = \Big\lceil\frac{2c_1}{\epsilon}\Big(2+\frac{\|f\|_{C^0}}{\beta}\Big) + d_f \Big\rceil$, $M = \| l - \beta V^\star + \nabla V^\star \cdot f   \|_{C^0(X\times U)} < \infty$, $r$ is defined in~(\ref{eq:rdef}), $c_3 = 2c_1c_2(2\beta + \|f\|_{C^0})/\beta $, the constant $c_1$ depends only on $V^\star$ and $X$ and $U$, whereas the constant $c_2$ depends only on sets $X$ and $U$.
\end{theorem}
\begin{proof}
According to Theorem~\ref{thm:bagby} and Lemma~\ref{lem:dp_bound} we can find a polynomial $\tilde{V}_{\tilde{d}}$ of degree no more than \[\tilde{d} = \Big\lceil\frac{2c_1}{\epsilon}\Big(2+\frac{\|f\|_{C^0}}{\beta}\Big)\Big\rceil  \] such that $\| V^\star - \tilde{V}_{\tilde{d}} \|_{C^1} \le \frac{\epsilon}{2} $ and such that $\tilde{V}_{\tilde{d}}$ satisfies the Bellman inequality~(\ref{eq:bellman}). Let $V$ be an arbitrary polynomial feasible in~(\ref{opt:lb_sos}) for some $d \ge 0$. Then
\begin{equation}\label{eq:bounds}
\| V^\star - V_d^\star \|_{L_1} \le \| V^\star - V \|_{L_1} \le \| V^\star - V \|_{C^0} \le \| V^\star - \tilde{V}_{\tilde{d}} \|_{C^0} +  \| V - \tilde{V}_{\tilde{d}} \|_{C^0} \le \frac{\epsilon}{2} + \| V - \tilde{V}_{\tilde{d}} \|_{C^0}.
\end{equation}
Hence, the goal is to find a degree $d \ge 0$ and a polynomial $V$ feasible in~(\ref{opt:lb_sos}) for that $d$ satisfying $\| V - \tilde{V}_{\tilde{d}} \|_{C^0} \le \epsilon/2$. Setting $V := \tilde{V}_{\tilde{d}} - \epsilon/2$, we clearly have $\| V - \tilde{V}_{\tilde{d}} \|_{C^0} \le \epsilon/2$; in addition, using Lemma~\ref{lem:shift} we know that
\begin{equation}\label{eq:lowerBound}
l - \beta V + \nabla V \cdot f \ge  \frac{1}{2}\beta\epsilon > 0
\end{equation}
and hence $V$ strictly satisfies the Bellman inequality and as a consequence of the Putinar's Positivstellensatz~\cite{putinar93} there exists a degree $d \ge 0$ such that $V$ is feasible in~(\ref{opt:lb_sos}). To bound the degree $d$ we apply the bound of Theorem~\ref{thm:nie} on $p :=  l - \beta V + \nabla V \cdot f$. From~(\ref{eq:lowerBound}) we know that $p_{\mathrm{min}} \ge \frac{1}{2}\beta\epsilon $. Next, we need to bound $\| p\|_{\mathbb{R}[x,u]}$ by bounding $\|  p \|_{C^0(X\times U)}$ and using Corollary~\ref{cor:norms}. We have
\begin{align*}
\| p\|_{C^0} &=\| l - \beta V + \nabla V \cdot f \|_{C^0} =  \| l - \beta \tilde{V}_{\tilde{d}} + \nabla \tilde{V}_{\tilde{d}} \cdot f + \frac{1}{2}\beta\epsilon \|_{C^0} \\ 
& \le  \| l - \beta V^\star + \nabla V^\star \cdot f +  \|_{C^0} + \beta\|V^\star - \tilde{V}_{\tilde{d}} \|_{C^0} + \| V^\star- \tilde{V}_{\tilde{d}} \|_{C^1}\|f\|_{C^0} + \frac{1}{2}\beta\epsilon \\
& \le M + \frac{1}{2}\beta\epsilon + \frac{1}{2}\epsilon\| f\|_{C^0} +\frac{1}{2}\beta \epsilon = M + \beta\epsilon + \frac{1}{2}\epsilon \|f \|_{C^0}.
\end{align*}
Finally, we need to estimate the degree of $p$. We have (assuming without loss of generality that $\tilde{d} + d_f -1 \ge \mathrm{deg}(l) $)
\[
\mathrm{deg}(p) =  \mathrm{deg} \big(l - \beta V + \nabla V \cdot f  \big) \le \tilde{d} + d_f - 1 \le \Big\lceil\frac{2c_1}{\epsilon}\Big(2+\frac{\|f\|_{C^0}}{\beta}\Big) + d_f \Big\rceil
\]
Setting $d_p := \Big\lceil\frac{2c_1}{\epsilon}\Big(2+\frac{\|f\|_{C^0}}{\beta}\Big) + d_f  \Big\rceil$ and using Theorem~\ref{thm:nie} and Corollary~\ref{cor:norms}, we conclude that for
\[
d \ge  c_2 \cdot \mathrm{exp}\Big[     \Big ( 2d_p^2 (n+m)^{d_{p}}\frac{k(d_p)(M + \beta\epsilon + \delta_1 \|f \|_{C^0})}{\beta\epsilon}  \Big)^{c_2}     \Big],
\]
the polynomial $V$ is feasible in~(\ref{opt:lb_sos}). Since $\|\tilde{V}_{\tilde{d}} - V \|_{C^0} \le \frac{\epsilon}{2}$, we conclude from~(\ref{eq:bounds}) that $\| V^\star - V_d^\star \|_{L_1} \le \epsilon$.
Inserting the expression for $k(d) = 3^{d+1}r^d $ from Corollary~\ref{cor:norms} yields~(\ref{eq:nonasympt}) and carrying out asymptotic analysis for $\epsilon \to 0$ yields
\[
d \ge O\Big(  \mathrm{exp}\big[\frac{1}{\epsilon^{3c_2}}(3(n+m)r)^{\frac{c_3}{\epsilon}}  \big]\Big),
\]
which is~(\ref{eq:asympt}).


 

\end{proof}

\begin{corollary}\label{rate}
It holds $\|V^\star - V^\star_d \|_{L_1(\mu_0)} = O(1/\log\log\,d)$.
\end{corollary}
\begin{proof}
Follows by inverting the asymptotic expression (\ref{eq:asympt}) using the fact that \[(3(n+m)r)^{\frac{2c_3}{\epsilon}} \ge \frac{1}{\epsilon^{3c_2}} (3(n+m)r)^{\frac{c_3}{\epsilon}}\] for small $\epsilon$.
\end{proof}

\section{Discussion}
The bound on the convergence rate $O(1/\log\log\, d)$ should be compared with the bound $O(1/\displaystyle\sqrt[c_2]{\log\,d})$ derived in~\cite{ns07} for \emph{static} polynomial optimization problems (here $c_2 \ge 1 $ is the, in general, unknown constant from Theorem~\ref{thm:nie}). The additional logarithm appearing in our bound seems to be unavoidable due to fundamental results of approximation theory (known as Bernstein inequalities) implying that Lipschitz continuous functions cannot be approximated by polynomials with rate faster than $1/d$ (in the sense that there exists a Lipschitz continuous function whose best degree-$d$ approximation converges to $f$ with the rate exactly $C / d$, $C > 0$, in the supremum norm on $[-1,1]^n$); this implies the $1/\epsilon$ dependence of $d_p$ from Theorem~\ref{mainresult} which then propagates to doubly exponential dependence on $1/\epsilon$ through Theorem~\ref{thm:nie}.

Therefore the primary point of improvement of the bound from Theorem~\ref{mainresult} and Corollary~\ref{rate} is the fundamental bound of Theorem~\ref{thm:nie} derived in~\cite{ns07}. As the authors of~\cite{ns07} remark, this bound is far from tight, at least in two special cases: the univariate case (i.e., $n+m = 1$ in our setting) or the case of a single constraint defining $X\times U$. In these cases the exponential in~(\ref{eq:schweigBound}) can be dropped, which results in $O(1/\log\, d)$ asymptotic rate of convergence in Corollary~\ref{rate}. In the general case, however, it is unknown whether the exponential in (\ref{eq:schweigBound}) can be dropped or whether the bound~(\ref{eq:schweigBound}) can be improved otherwise~\cite{ns07}.

\section{Extensions}~\label{sec:extension}
The approach for deriving this bound can be extended to other settings. In particular, similar bounds, with identical $O(1/\log\log\,d)$ asymptotics, hold for the finite-horizon version of the problem, both in continuous and discrete time, as well as for the discounted discrete-time infinite-horizon variant (the former was treated using the moment-sum-of-squares approach, in continuous time, in~\cite{lhpt08} and the latter was treated in~\cite{savorgnan}). The derivation in discrete-time is completely analogous and the results hold under milder assumptions (Assumption~1 $(d)$ can be replaced by $V^\star$ Lipschitz and Assumption 1 $(e)$ can be dropped completely). For the finite-horizon continuous-time problem, the only difference is in Lemma~\ref{lem:dp_bound}, where the constant shift $\tilde{V}(x) = \hat{V}(x) - a$,  is replaced by the affine shift $\tilde{V}(t,x) = \hat{V}(t,x) - a - b(T-t)$ for suitable $a > 0$, $b >0$ ensuring that $\tilde{V}$ satisfies the corresponding finite-time Bellman inequality and its boundary condition (hence the two degrees of freedom).

\end{document}